\documentclass[letterpaper, 10pt, conference]{ieeeconf}
\IEEEoverridecommandlockouts
\usepackage{graphicx,epstopdf,multirow,multicol,pgfplots}
\usepackage{amsmath}
\usepackage{float}
\usepackage{bm}
\usepackage{cite}
\usepackage{breqn}
\usepackage{amssymb,amsthm}
\usepackage[T1]{fontenc}
\usepackage[colorlinks]{hyperref}

\usepackage{booktabs}
\usepackage{tikz,booktabs}
\usetikzlibrary{calc}
\allowdisplaybreaks
\usepackage{xspace}
\newcommand{\ie}{\emph{i.e.}\xspace}
\newtheorem{proposition}{Proposition}
\newtheorem{lemma}{Lemma}
\newtheorem{corollary}{Corollary}

\usepackage{algorithm}
\usepackage[noend]{algpseudocode}
\algnewcommand\algorithmicinput{\textbf{Input:}}
\algnewcommand\Input{\item[\algorithmicinput]}

\algnewcommand\algorithmicoutput{\textbf{Output:}}
\algnewcommand\Output{\item[\algorithmicoutput]}

\title{\Large \bf
Landmark Placement for Localization in a GPS-denied Environment
}

\author{Kaarthik Sundar$^{\dagger}$\thanks{$^{\dagger}$ Center for Nonlinear Studies, Los Alamos National Laboratory, NM.},\;
Shriram Srinivasan$^{\dagger}$,\;
Sohum Misra$^{*}$\thanks{$^{*}$ Dept. of Aerospace Engg., University of Cincinnati, OH. \texttt{}},\;
Sivakumar Rathinam$^{\ddagger}$\thanks{$^{\ddagger}$ Dept. of Mechanical Engg., Texas A\&M University,
College Station, TX.},\;
Rajnikant Sharma$^{*}$} \;

\renewcommand{\baselinestretch}{1.15}
\pgfplotsset{compat=1.13}

\begin{document}

\maketitle
\thispagestyle{empty}
\pagestyle{empty}

\begin{abstract}
Path planning algorithms for unmanned aerial or ground vehicles, in many surveillance applications, rely on Global Positioning System (GPS) information for localization. However, disruption of GPS signals, by intention or otherwise, can render these plans and algorithms ineffective. This article provides a way of addressing this issue by utilizing stationary landmarks to aid localization in such GPS-disrupted or GPS-denied environment. In particular, given the vehicle's path, we formulate a landmark-placement problem and present algorithms to place the minimum number of landmarks while satisfying the localization, sensing, and collision-avoidance constraints. The performance of such a placement is also evaluated via extensive simulations on ground robots.
\end{abstract}


\section{Introduction} \label{sec:intro}
Unmanned Vehicles (UVs) are routinely used in civilian and military applications pertaining to surveillance and data gathering missions in civil infrastructure management, traffic management, crop monitoring \cite{Corrigan2007,Manyam2016}, to name a few. Over the past decade, a wide range of path planning, routing, and motion planning algorithms have been developed for UVs to handle heterogeneity \cite{Sundar2013,Levy2014,Sundar2015,Sundar2016}, resource constraints \cite{Rathinam2007}, fuel constraints \cite{Sundar2014,Sundar2012,Sundar2017a,Sundar2016a} etc. Most of these algorithms rely on GPS information to aid the vehicles to stay on its route. Disruption of the GPS signals, deliberate or otherwise can render these algorithms ineffective. For instance, entire cities in the states of Arizona and California experienced unintentional GPS interference \cite{Carroll2003,Hoey2005} for a period of five and thirty seven days respectively. This article presents a technique to aid localization of UVs in such GPS-denied/GPS-restricted environments. Localization entails the determination of the position and heading of a UV from the available measurements. Our approach assumes that routes for the UVs are decided \emph{apriori} using  various path planning and routing algorithms and cannot be changed. 

The routes for the UVs are rendered useless without localization capability, \ie,  knowledge of its  position and heading. The usual practice is to combine GPS and measurements obtained from Inertial Measurement Units (IMUs) with traditional filtering techniques to estimate the position and heading of the UVs. In GPS-denied environments, range and bearing sensors yield relative position and bearing measurements from known landmarks which are in turn used to estimate the states (position and heading) of the UVs; the state-estimation problem is then solved using an Extended Kalman Filter (EKF) or its information counterpart, the Extended Information Filter (EIF). It is known that the state estimates obtained using the estimation algorithm would provide consistent and bounded estimates only if enough measurements are available \cite{Song1992} or rather, in the parlance of control theory, if the the system is observable. If the system is observable and linear, a bound on the uncertainty in the estimation error can be related to the the eigenvalues of the observability grammian \cite{Song1992}. For the problem of localization of UVs, authors in \cite{Sharma2012} have shown that each UV requires bearing measurements from two distinct landmarks for observability of the system; the observability of the system is shown using Lie derivative theory. An important problem that arises in this context is that of landmark placement, \ie, given a route for the UV, to place landmarks so that the system is always observable; this would in turn ensure the error in state estimates of the position and heading of the UV, obtained using the estimation algorithms, remain bounded \cite{Song1992}. It is trivial to do so by placing finely-spaced landmarks on a grid of the area that the UV is monitoring. Nevertheless, if the area is large, then this can result in too many landmarks being placed. Furthermore, the sensor on the vehicle that measures the bearing from a given landmark is usually a camera which has its own restricted field of view. This article considers such a landmark placement problem that aims to place the minimum number of landmarks to ensure observability of the system at all times, while taking into account a collision avoidance constraint and the restricted field of view of the camera. The Fig.~\ref{fig:illustration} shows the path taken by the UV and  the placement of landmarks to ensure observability of the vehicle while accounting for the restricted field of view of the bearing sensor (a camera).

\begin{figure}
    \centering
    \includegraphics{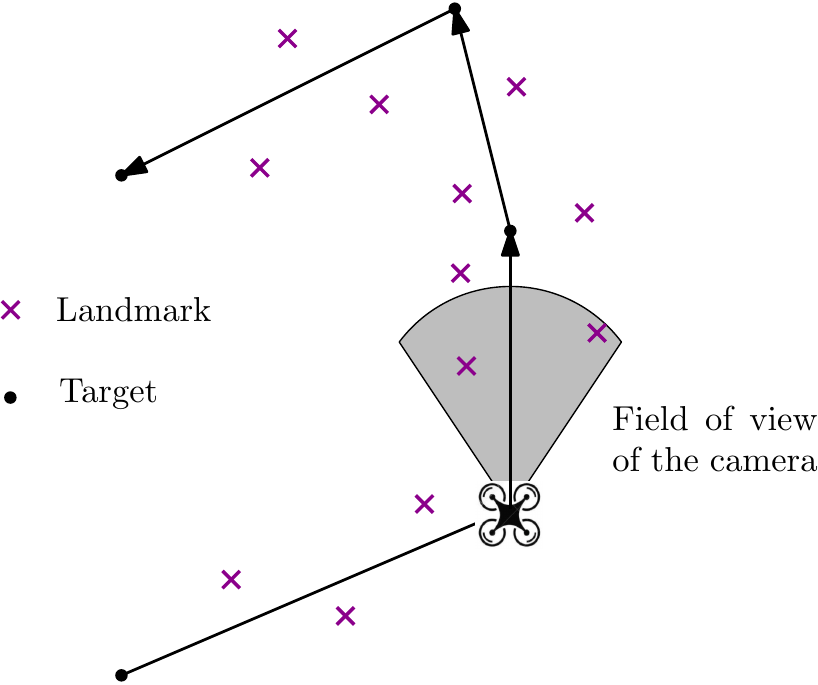}
    \caption{A placement of landmarks to facilitate observability, and hence, accurate localization for the UV.}
    \label{fig:illustration}
\end{figure}

\subsection{Related Work} \label{subsec:lit_review}
The idea of using landmarks to aid in localization has been addressed previously in  \cite{Sharma2012, Sundar2017}. There are also a number of techniques that use so-called proxy landmarks for localization. Proxy landmarks can refer to either the use of additional localization vehicles or radio signals received from suitably positioned neighbouring vehicles \cite{Kurazume1998,Sanderson1997,Roumeliotis2000}. In  \cite{Rathinam2015}, the authors develop a fast $13$-approximation algorithm to place landmarks to make the system observable given the path for the UVs. They assume that the bearing sensor's  (camera's) field of view is circular and formulate the problem as a generalization of a unit-disc cover problem \cite{Liu2014,Biniaz2015}. To the best of our knowledge, the previous works in the literature either neglect  the constraints imposed by the bearing sensor \cite{Sharma2012} or assume that the field of view is circular \cite{Rathinam2015}, both of which are unrealistic.  While there is theoretically the possibility of achieving a circular field of view using multiple cameras, obtaining bearing information from them can significantly increase the overhead on the image processing algorithms required to do so. 

We aim to remedy these shortcomings and accordingly, we rule out the possibility of having multiple cameras (and circular field of view) right at the outset and instead take into account the constraints imposed by the bearing sensor. The other major point of difference  from previous work  is that we assume a given discrete set of potential landmark locations rather than allowing the landmarks to be placed anywhere in the area of interest. Practically, our assumption is more realistic  because there are often regions in the area of interest where  terrain restrictions prevent placement of landmarks. In the following section we formally define the problem and detail the main contributions of this article.

\subsection{Landmark Placement Problem (LPP)} \label{subsec:lpp} 
As mentioned in Sec. \ref{sec:intro}, for observability, a UV requires bearing measurements from at least two distinct landmarks at any point in time. In this article, this condition is enforced as a graph-theoretic requirement to place the landmarks in the following landmark placement problem (LPP):

\noindent \textit{Given: A UV equipped with a camera possessing a viewing range of $R$ units and viewing angle of $\alpha (\leqslant 180^{\circ})$, a set of targets and a path taken by the UV from a source target to the destination target that visits every target exactly once, and a set of potential landmark locations where landmarks can be placed. \\
Objective: Placement of landmarks in a subset of potential locations such that the following conditions are satisfied -- (1) at least two landmarks have to be in the field of view of the UV's camera at any point in the UV's path, (2) two landmarks  cannot be coincidental from the point of view of the camera, (3) the distance between the UV and any landmark should be greater than $p$ units, and (3) the number of landmarks placed is a minimum.}

\noindent A feasible solution to the LPP is shown in Fig. \ref{fig:illustration}. The condition (1) enforces observability of the system and hence guarantees convergence of any estimation algorithm (EKF or EIF) to estimate the position and heading of the UV while the condition (2) and (3) are collision-avoidance constraints.

\subsection{Contributions} \label{subsec:contributions}
The contributions of this article are as follow: (1) the LPP is formulated as a graph theoretic problem, (2) a decomposition of the problem into smaller sub-problems (one for each edge) in the UV's path is presented, (3) a greedy algorithm is then developed  for optimal solution of each sub-problem, and finally (4) the solutions to the sub-problems are put together to construct a feasible solution to the LPP. The final placement of landmarks obtained via the proposed algorithm is subjected to extensive simulation experiments. In subsequent sections, we present the algorithm to solve the LPP followed by simulation on ground robots.   

\section{Algorithms and Analysis} \label{sec:algo}
Any point in the path of the UV is said to be covered by a landmark $k$ if $k$ is in the field of view of the UV's camera. Let $L = \{\ell_1, \dots, \ell_m\}$ denote the set of potential landmark locations and $T = \{t_1, \dots, t_n\}$, the set of targets. We remark that the set $L$ is the reduced set of potential landmark locations after preprocessing out those locations which do not satisfy the collision avoidance constraints. Without loss of generality, we assume that $t_1$ and $t_n$ are the source and destination targets where the UV's path begins and ends, respectively. Also, let the UV's path be denoted by $\{(t_1, t_2), (t_2, t_3), \dots, (t_{n-1}, t_n)\}$, where $(t_i, t_{i+1})$ is the $i$\textsuperscript{th} edge in the path. The number of edges in the UV's path is $(n-1)$. We shall say that a set of landmarks $L_1 \subseteq L$ covers an edge $(t_i, t_{i+1})$ if for any point on the edge, there exist at least two landmarks that cover the point. The objective of the LPP is to cover every edge in the UV's path. This allows for an edgewise decomposition. Hence, we will first present an algorithm to find the optimal number of landmarks and their placements to cover a given edge in the path of the UV.  

\subsection{An optimal landmark placement for an edge} \label{subsec:algo_edge}
Let $e = (t_i, t_{i+1})$ denote any edge in the UV's path. The objective is to solve the LPP for this edge $e$. To that end, let $(x_i, y_i)$ and $(x_{i+1}, y_{i+1})$ denote the coordinates of the targets $t_i$ and $t_{i+1}$, respectively. Without loss of generality, we assume that $(x_i, y_i) = (0,0)$ and $(x_{i+1}, y_{i+1}) = (d, 0)$. This is possible by a suitable rigid-body transformation (translation \& rotation) of the coordinate axes $\mathcal R_i$. For the sake of completeness, we present the coordinate transformations for this particular case.
\begin{enumerate}
    \item Translate the coordinate axis so that the origin coincides with $(x_i, y_i)$. This would result in the new coordinates of $t_i$ and $t_{i+1}$ being the origin and $(x_{i+1}-x_i, y_{i+1}-y_i)$, respectively.
    \item Rotate the coordinate axis by an angle $$\theta = \tan^{-1}\frac{y_{i+1}-y_i}{x_{i+1}-x_i};$$ such that the rotation aligns the positive direction of the $X$-axis with the edge $(t_i, t_{i+1})$ and let the new reference axes be denoted by $\mathcal R_f$. Also, let the coordinate of $t_{i+1}$ with respect to the new reference axes $\mathcal R_f$ be denoted by $(d,0)$ where, $d$ is the length of the edge.
\end{enumerate}
The coordinates of every potential landmark location $\ell \in L$ with respect to the new reference axes  $\mathcal R_f$ is also computed. Now, given that the edge is aligned along the positive $X$-axis of $\mathcal R_f$ with $t_i = (0,0)$ and $t_{i+1} = (d, 0)$, and the coordinates of every $\ell \in L$ in $\mathcal R_f$, we present an algorithm to check if a landmark placed at a location $\ell_k \in L$ for any $k=1, \dots, m$, would be within the camera's field of view for some part of the edge the UV traverses, and if so, compute that part of the edge. The following lemma gives a necessary condition for a landmark placed at a location $\ell_k \in L$ to be within the field of view of the camera as the UV traverses some part of the edge $e$. 

\begin{lemma} \label{lem:necessary}
Consider any potential landmark location $\ell_k \in L$. $\ell_k$ would be in the field of view of the camera as the UV traverses some part of the edge from $t_i$ to $t_{i+1}$ only if the coordinates, $(x_k, y_k)$, of $\ell_k$ in the new reference axes $\mathcal R_f$ lie within the region defined by 
\begin{enumerate}
    \item $|y_k| \leqslant R \sin (\alpha/2)$, 
    \item $x_k > 0$,
    \item $-\alpha/2 \leqslant \tan^{-1}(y_k/x_k) \leqslant \alpha/2$, and
    \item if $x_k \geqslant d$, then $(x_k-d)^2 + y_k^2 \leqslant R^2$
\end{enumerate}
where, $R$ is the camera's viewing range and $\alpha$ is the camera's viewing angle.
\end{lemma}
\begin{proof}
The proof is direct by making the observation that the constraints define the region in Fig~\ref{fig:infuence_region}. We shall refer to this region as the ``region of influence''. If the landmark $\ell_k$ is placed within this region of influence, then it will be in the field of view of the UV's camera as the vehicle traverses some part of the edge.
\begin{figure}
    \centering
    \includegraphics[scale=0.7]{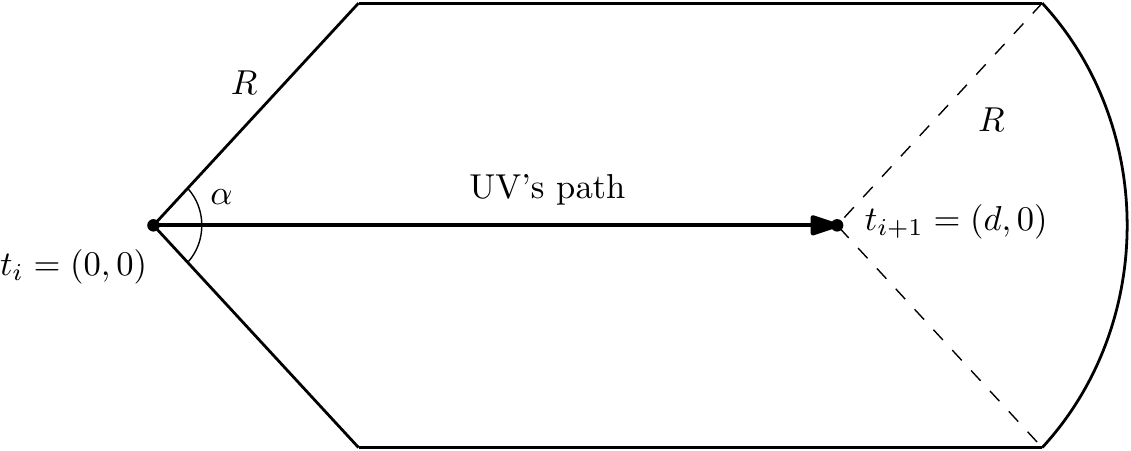}
    \caption{The region defined by the conditions in Lemma \ref{lem:necessary}.}
    \label{fig:infuence_region}
\end{figure}
\end{proof}

The next lemma gives the interval for which a landmark placed in the region of influence would be in the camera's field of view.

\begin{lemma} \label{lem:interval}
Consider a potential landmark location $\ell_k \in L$ whose coordinate with respect to $\mathcal R_f$, $(x_k, y_k)$, is within the region of influence, as defined by Lemma \ref{lem:necessary}. Then $\ell_k$ would be in the camera's field-of-view as the UV travels from the point $t^s$ to $t^d$, where the coordinates of $t^s$ and $t^d$ in $\mathcal R_f$ are given by
\begin{flalign}
t^s &= \left(\max\left\{x_k-\sqrt{ R^2-y_k^2}, 0\right\}, 0\right), \label{eq:ts} \\
t^d &= \left(\min\left\{x_k-|y_k|\operatorname{cot}\left(\frac{\alpha}2\right), d\right\}, 0\right). \label{eq:td}
\end{flalign}
\end{lemma}
\begin{proof}
Let $(x_s, 0)$ be the position of the UV when the landmark placed at the location $\ell_k \in L$ enters the field of view of the camera. The Fig.~\ref{fig:xi} illustrates this case. 
\begin{figure}[!h]
    \centering
    \includegraphics[scale=0.8]{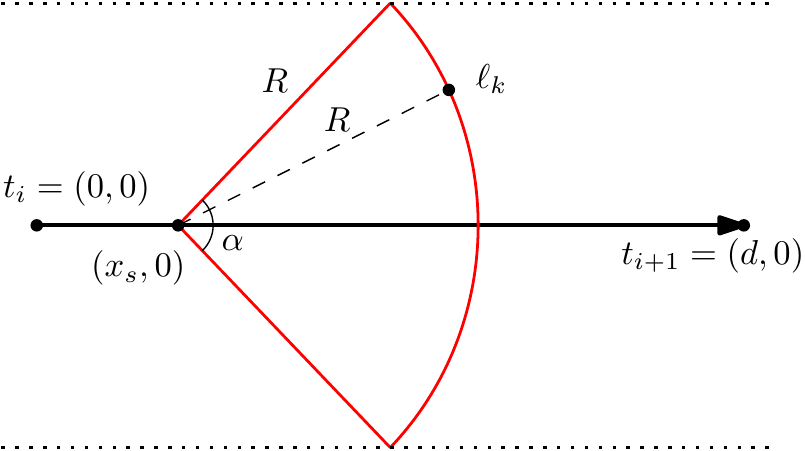}
    \caption{$\ell_k$ entering the camera's field of view.}
    \label{fig:xi}
\end{figure}
Then, it is clear from \ref{fig:xi} that the distance between the points $(x_s,0)$ and $(x_k, y_k)$ is exactly $R$ units, which yields $x_s = x_k - \sqrt{ R^2-y_k^2}$. Hence, the position of the vehicle  $t^s$ when the landmark at $\ell_k$ enters the camera's field of view is given by the Eq. \eqref{eq:ts}. 

Similarly, let $(x_d, 0)$ be the position of the UV when the landmark placed at the location $\ell_k \in L$ leaves the field of view of the camera. The Fig.~\ref{fig:xf} illustrates this case. For ease of exposition, we shall refer to $(x_d, 0)$ as $A$ and $(x_k,0)$ as $B$. The perpendicular distance between $\ell_k$ and the edge is given by $|y_k|$. Then, the length of the segment from $A$ to $B$ is $|y_k|\operatorname{cot}(\alpha/2)$. This in turn implies that $x_d = x_k - |AB| = x_k - |y_k|\operatorname{cot}(\alpha/2)$, proving the claim.
\begin{figure}[!h]
    \centering
    \includegraphics[scale=0.8]{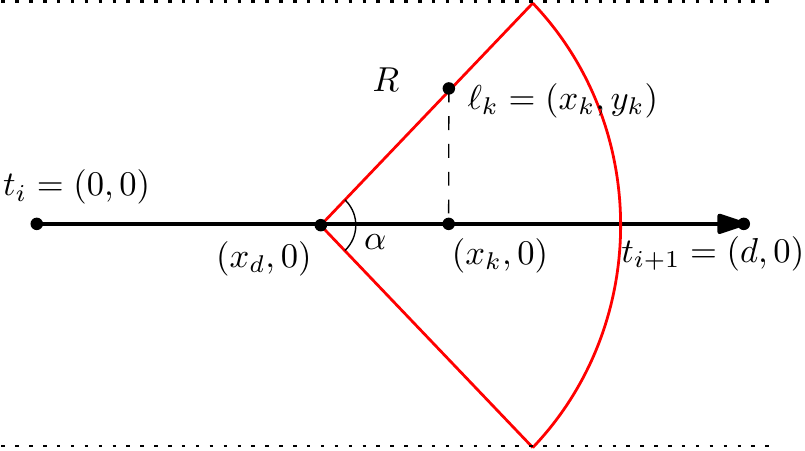}
    \caption{$\ell_k$ leaving the camera's field of view.}
    \label{fig:xf}
\end{figure}
\end{proof}

\begin{corollary} \label{cor:circ}
If the field of view of the camera is circular then, 
\begin{flalign}
t^s &= \left(\max\left\{x_k-\sqrt{ R^2-y_k^2}, 0\right\}, 0\right), \label{eq:ts_c} \\
t^d &= \left(\min\left\{x_k+\sqrt{ R^2-y_k^2}, 0\right\}, 0\right). \label{eq:td_c}
\end{flalign}
\end{corollary}
The Lemmas \ref{lem:necessary} and \ref{lem:interval} together give the conditions under which a landmark placed at a location $\ell_k \in L$ for any $k = 1,\dots,m$ can cover some part of the edge and if so, allows the computation of the interval (part of the edge) that it can cover. This computation can be repeated for every potential landmark location $\ell \in L$ to obtain the interval defining a part of the edge that can be covered by a landmark placed at $\ell$; let $I_{\ell}$ denote the interval defined by Lemma \ref{lem:interval} for the landmark location $\ell \in L$. We remark that $I_{\ell}$ can be empty if the coordinates of $\ell$ does not satisfy the necessary condition of Lemma \ref{lem:necessary}. Given these intervals, we now develop a greedy algorithm to find the minimal set of locations $L_1 \subseteq L$, where landmarks can be placed to cover the entire edge. This is a generalization of the well-studied ``Interval Covering'' problem and is formally stated as follows: given an interval $[a, b]$ and a set of intervals $[a_i, b_i]$, for $i\in \{1, \dots, n\}$, choose a minimal subset $S \subset\{1, \dots,n\}$ such that $[a,b] = \bigcup_{i\in S} [a_i, b_i]$. The problem of covering an edge with a set of landmarks can be considered as a 2-interval covering problem that seeks a subset of intervals to enclose all the points in the edge at least twice. A pseudocode of the algorithm is given in Algorithm \ref{algo:pseudocode}. 

\begin{algorithm}
\caption{Greedy algorithm}\label{algo:pseudocode}
\begin{algorithmic}[1]
\vspace{1ex}
\Input $I_{\ell} = [a_{\ell}, b_{\ell}] \,\,\forall \ell \in L$ and edge $e=(t_i, t_{i+1})$
\Output $L_1 \subseteq L$ such that $L_1$ covers the edge $e$
\State $I \gets \bigcup_{\ell \in L} I_{\ell}$
\State $L_1 \gets \emptyset$
\State $c \gets 0$ \Comment{x-coordinate of $t_i$ in $\mathcal R_f$}
\State $\hat L \gets$ two intervals that contain $c$ with the largest and second largest endpoint
\State $c \gets \min\{b_{\ell}, \ell \in \hat L\}$, $c_1 \gets \max\{b_{\ell}, \ell \in \hat L\}$
\State $L_1 \gets L_1 \bigcup \hat L$ and $I \gets I \setminus \hat L$
\While{$c < d$} \Comment{check if the edge is covered}
    \State $\ell_1 \gets \operatorname{argmax}_{I_{\ell} \in I} \{b_{\ell} : I_{\ell} \text{ covers } c\}$, \label{step:prob}
    \State If Step \ref{step:prob} is infeasible; exit \Comment{problem infeasible}
    \State $L_1 \gets L_1 \bigcup \{\ell_1\}$ and $I \gets I \setminus \{\ell_1\}$
    \State $c \gets \min (c_1, b_{\ell_1})$, $c_1 \gets \max (c_1, b_{\ell_1})$ 
\EndWhile
\end{algorithmic}
\end{algorithm}

The Algorithm \ref{algo:pseudocode} is a greedy algorithm for a 2-interval covering problem which is also optimal. The intuition behind the algorithm is as follows: We shall use that analogy that the edge is a tunnel with an entrance and and exit and intervals as lamps with the length of the interval denoting its lighting range. The goal of the 2-interval cover problem is then to light the tunnel using the minimum number of lamps such that each point in the tunnel falls in the lighting range of at least two lamps. The algorithm first chooses two lamps which cover the entrance of the tunnel and has the longest and the second longest lighting range into the tunnel. Then, it continuously choose the lamp which overlaps the smaller of the two current lighting ranges and has the longest range further into the tunnel, until the exit is reached. 

\begin{proposition} \label{lem:greedy}
The greedy algorithm presented in Algorithm \ref{algo:pseudocode} is optimal for the 2-interval covering problem.
\end{proposition}
\begin{proof}
Consider an optimal set of intervals $L_1$ for the 2-interval covering problem and let the set $L_{t_1} \subseteq L_1$ denote the set of intervals covering the point $t_1 = (0,0)$ (the starting point) in the edge. Clearly, $|L_{t_1}| = 2$. If both these intervals from $L_{t_1}$ are replaced with two other intervals $[a_i, b_i]$ and $[a_j, b_j]$ from $L$ whose endpoint $b_i$ and $b_j$ are the maximal and second maximal, the remaining uncovered part of the edge, $(b_j,d]$, would be covered of the remaining intervals in the optimal solution \ie $L_1 \setminus L_{t_1}$. Hence, the remaining uncovered part of the edge can be covered with no more intervals than the analogous uncovered part from the optimal solution. Therefore, a solution constructed using the intervals $[a_i, b_i]$ and $[a_j, b_j]$ and the optimal solution for the remaining uncovered part of the edge, $L_1 \setminus L_{t_1}$ will also be optimal, proving the claim. 
\end{proof}

\subsection{Landmark placement for the entire path} \label{subsec:path}
The potential set of locations on which landmarks can be placed so that the UV can be localized at all time instants as it traverses its path can be solved by decomposing the problem over the edges. We remark that, despite the fact that the LPP is being solved to optimality over each edge, a solution obtained for the LPP for the entire path using the optimal solutions for each edge in the path need not be optimal for the entire path. This makes the algorithm for the entire UV path a heuristic approach. The solution to the LPP for the entire path is computed by taking a union of the potential locations obtained for every edge in the UV's path. In the next section, we evaluate the effectiveness of the proposed landmark placement with extensive simulations. 

We remark that the landmark placement for the entire path would not ensure complete observability at the corners. For instance consider the Fig.~\ref{fig:corner}. 
\begin{figure}[!h]
    \centering
    \includegraphics{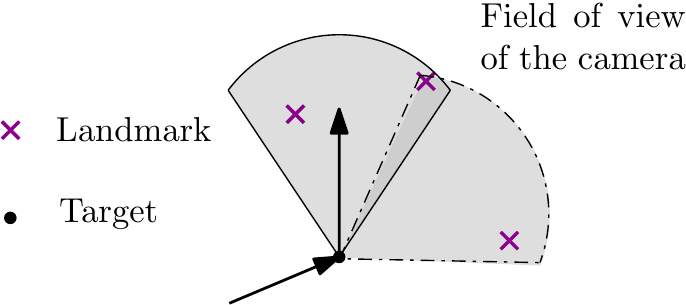}
    \caption{Observability at the  intersection of two edges}
    \label{fig:corner}
\end{figure}
Here, the UV would be observable only at the beginning of the turn and then at the end of the turn; this is assuming that the UV is a ground robot or a quadracopter that can make a turn around its axis while remaining stationary at a point. This is reasonable from a localization perspective because when the UV makes a turn, it remains stationary and the position of the UV does not change. This issue is averted if we assume the camera's field of view is circular (see Corollary \ref{cor:circ}), which is not practical. Nevertheless, if needed, the algorithm can be modified suitably to make the UV observable at all times in a corner. Due to page limitations, this modification to the algorithm is not discussed here and is delegated to future work. 

\section{Results} \label{sec:results}
\subsection{Simulation Results} \label{subsec:sim}
For the simulation experiments, two cases, one with 6 targets and another with 7 targets were considered. For both the cases, the path taken by the vehicle was computed \emph{a priori} using the Lin-Kernighan-Helsgaun (LKH) heuristic \cite{Lin1973}. The way points were generated randomly in a grid of size $4\times8$ . The viewing angle and range of the camera was set to $50^{\circ}$ and $2$ meters (m), respectively. A viewing angle and range were chosen to keep the simulation setup practical. A high controller gain was chosen to ensure a high turn rate for the vehicle. The set of grid points in a grid of size $0.5 \times 0.5$ was chosen as the set of potential landmark locations. The vehicle's velocity was kept constant at $0.2$ m/s. A $p$ value of $0.05$ m (collision avoidance distance) was used for the separation distance between the vehicle and any location where the landmarks were to be placed. Although this distance can be increased, a small value was chosen to make the instances feasible. The greedy algorithm presented in Sec. \ref{sec:algo} is then used to compute the subset of locations where landmarks are to be placed. Using the landmarks at the locations obtained from the algorithm, an EIF was used to estimate the position and heading of the vehicle using the bearing measurements. The results of both the simulations are shown in Fig. \ref{fig:exp_6WP_22LM_traj} and Fig. \ref{fig:exp_7WP_31LM_traj}.

\begin{figure}[h!]
    \centering
    \includegraphics[width = 85mm,keepaspectratio]{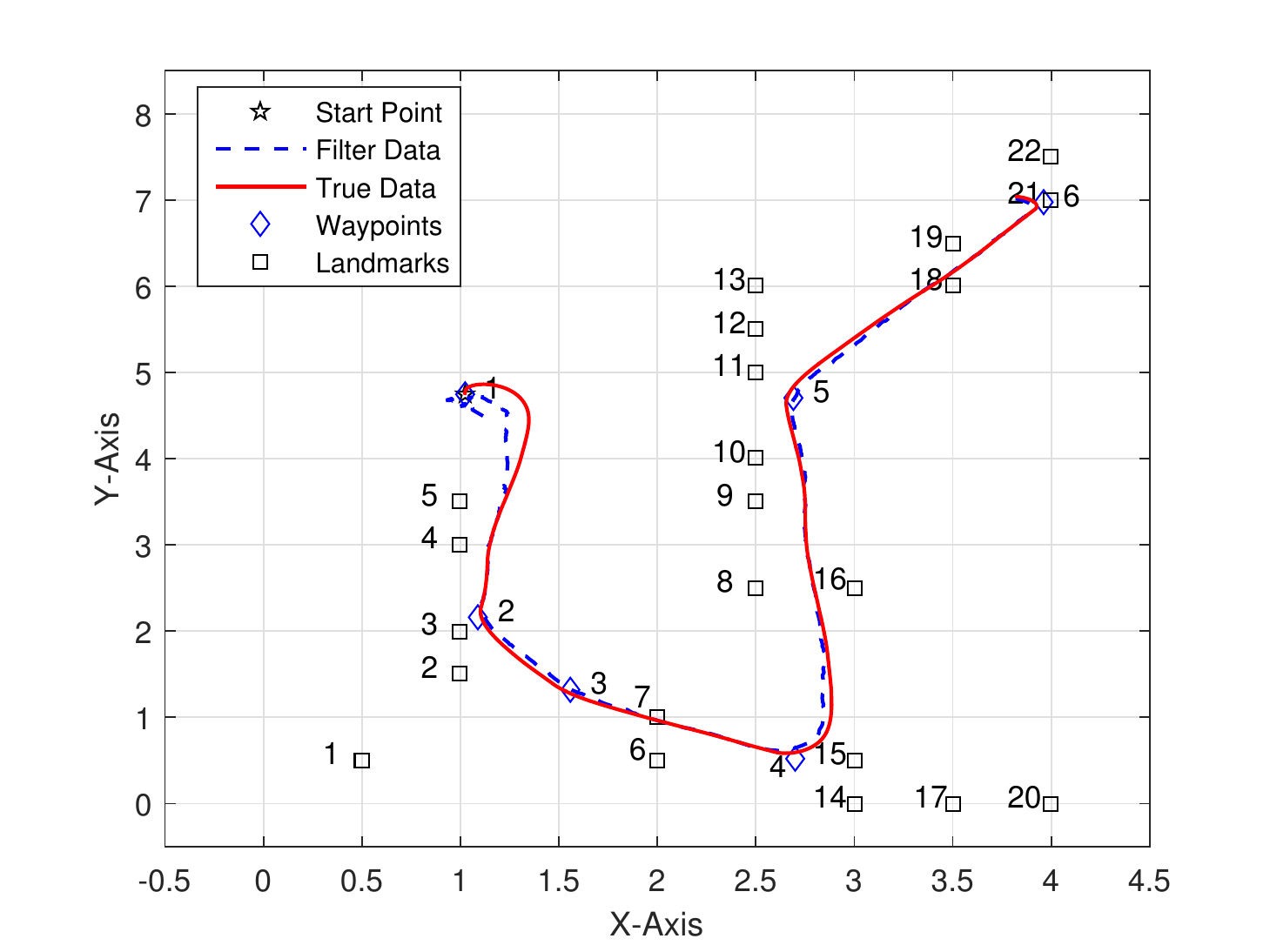}
    \caption{True and estimated trajectories of the vehicle with 6 targets (way points) and 22 landmarks.}
    \label{fig:exp_6WP_22LM_traj}
\end{figure}

\begin{figure}[h!]
    \centering
    \includegraphics[width = 85mm,keepaspectratio]{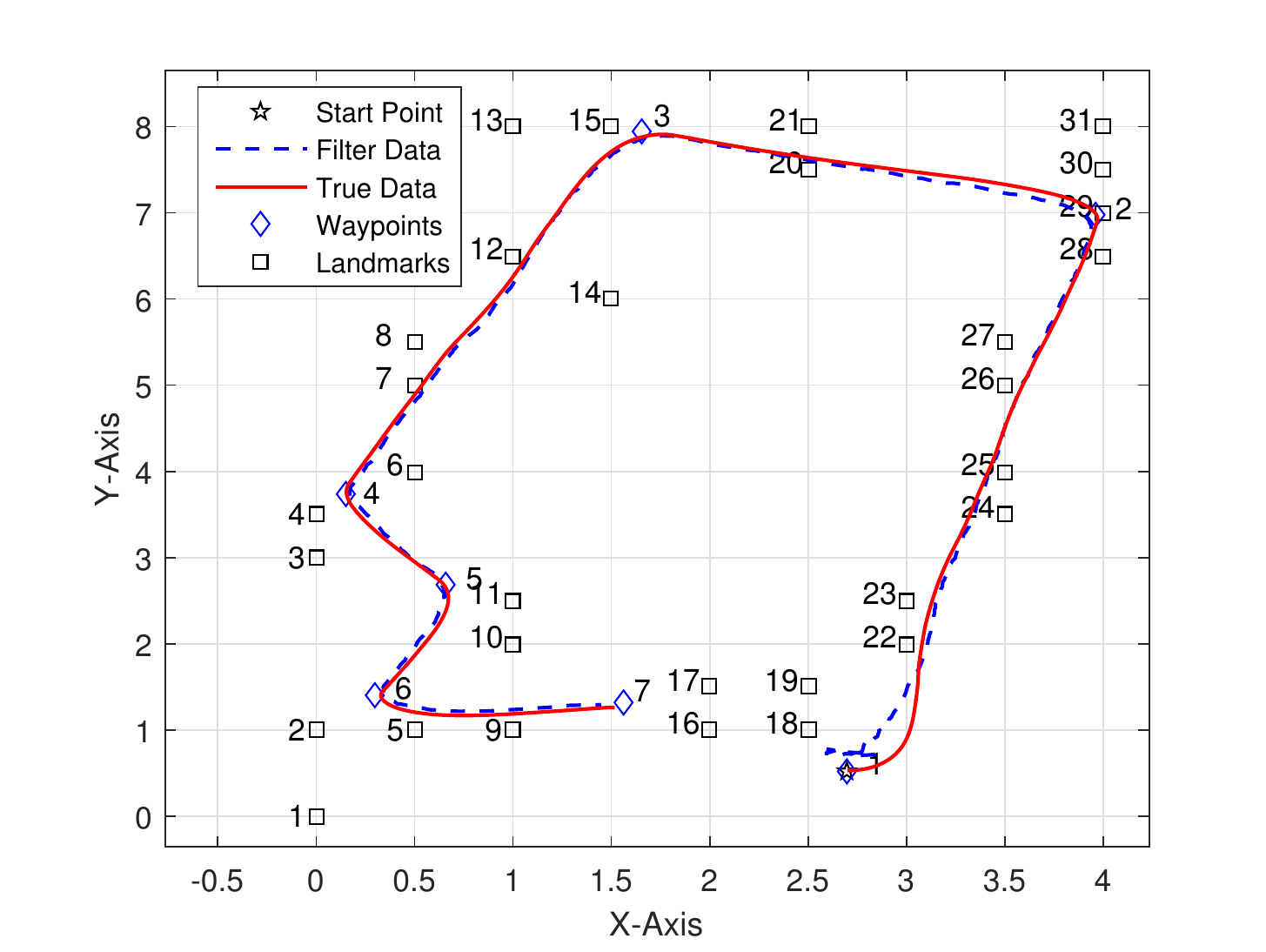}
    \caption{True and estimated trajectories of the vehicle with 7 targets (way points) and 31 landmarks. }
    \label{fig:exp_7WP_31LM_traj}
\end{figure}

The error in the position and the heading of the vehicle corresponding to  Fig. \ref{fig:exp_6WP_22LM_traj} and Fig. \ref{fig:exp_7WP_31LM_traj} are shown in Fig. \ref{fig:exp_6WP_22LM_err} and Fig. \ref{fig:exp_7WP_31LM_err}, respectively. It can be seen that the error lies within the $3\sigma$ bound most of the time, hence ensuring that the placement of landmarks obtained from the greedy algorithm is sufficient for localization of the vehicle. 

\begin{figure}[!h]
    \centering
    \includegraphics[width = 85mm,keepaspectratio]{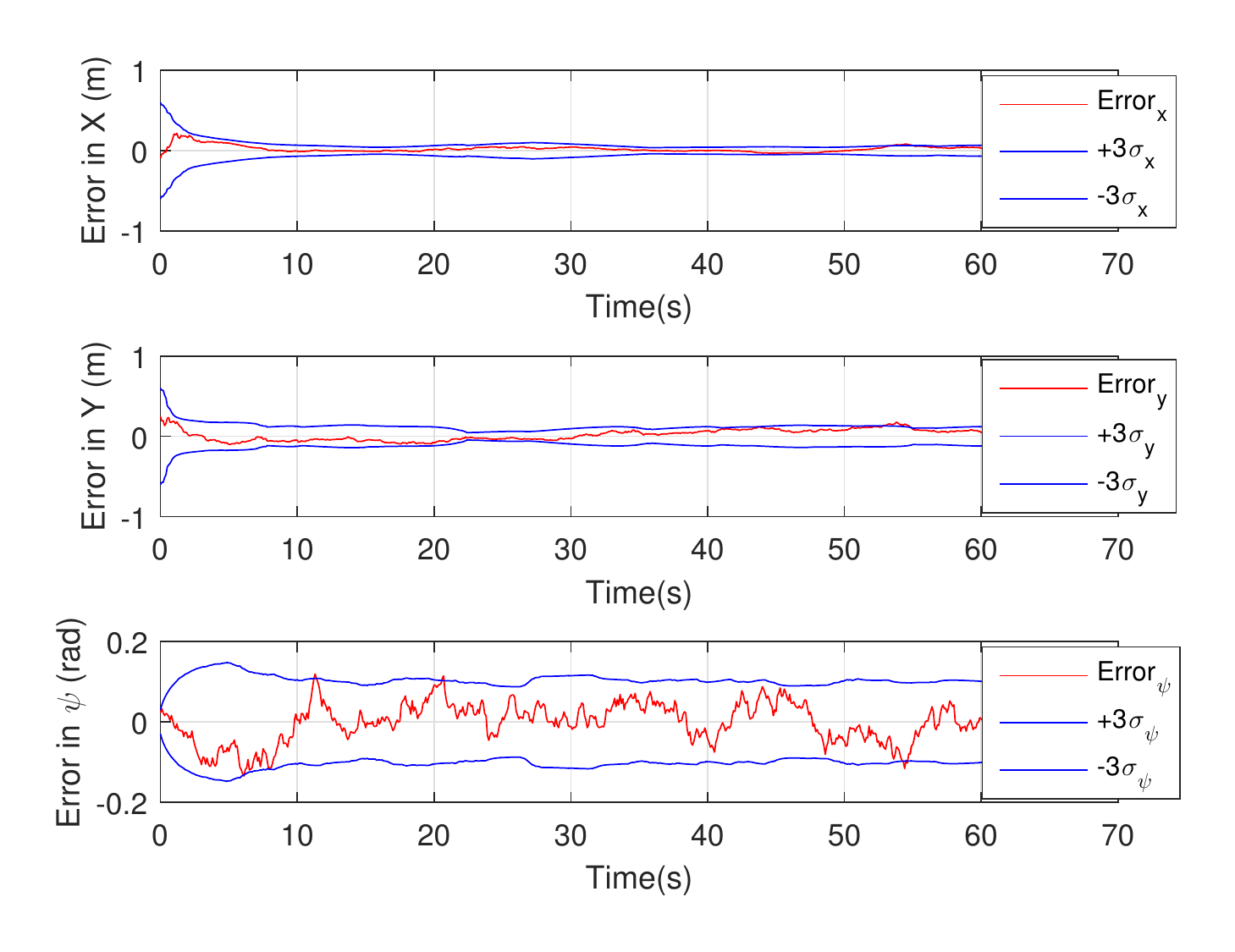}
    \caption{Error in X, Y and $\psi$ (heading) of the vehicle with 6 targets (way points) and 22 landmarks.}
    \label{fig:exp_6WP_22LM_err}
\end{figure}

\begin{figure}[!h]
    \centering
    \includegraphics[width = 85mm,keepaspectratio]{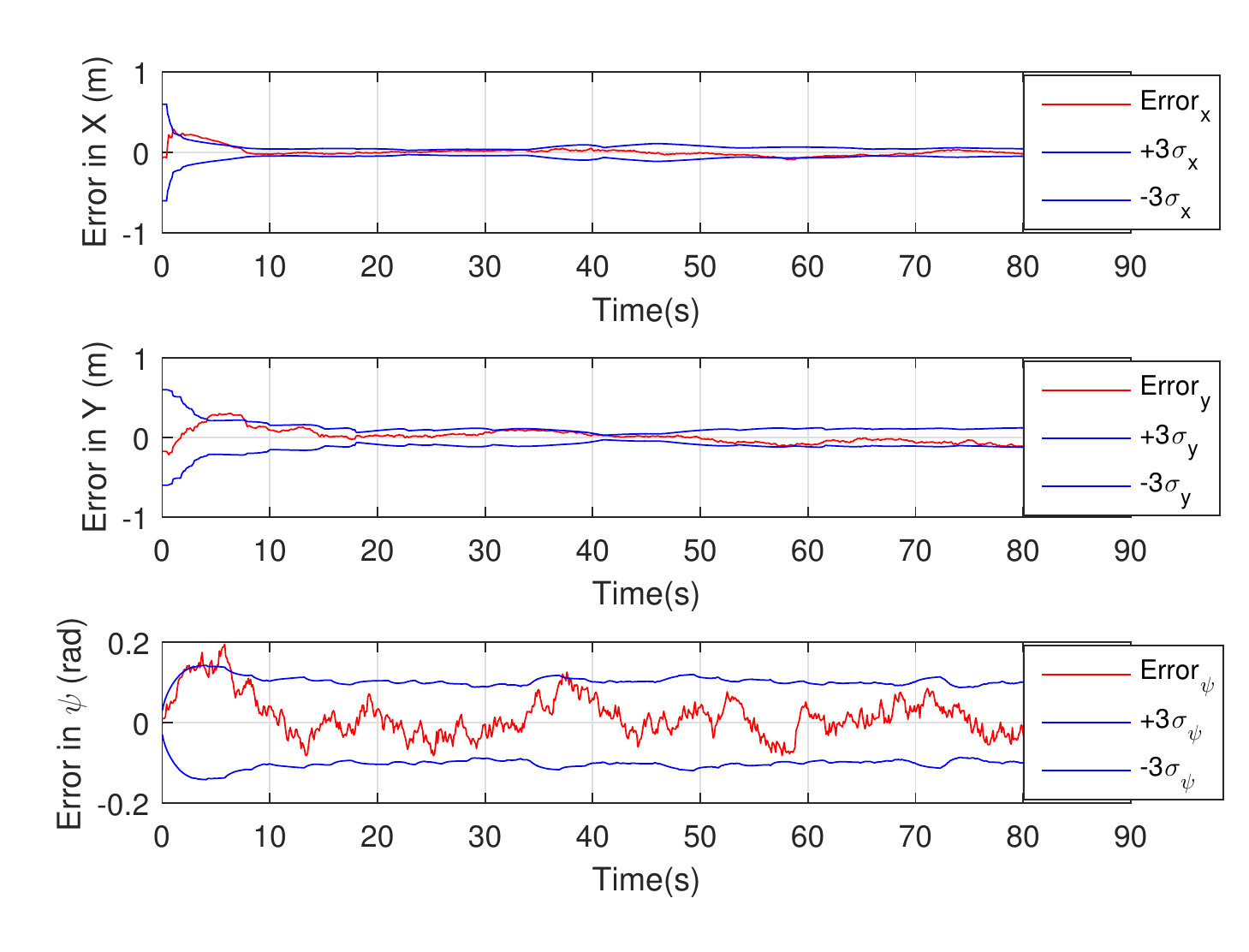}
    \caption{Plot showing the error in X, Y and $\psi$ (heading) of the vehicle with 7 targets (way points) and 31 landmarks.}
    \label{fig:exp_7WP_31LM_err}
\end{figure}

\subsection{Performance of the greedy algorithm} \label{subsec:algo_results}
In this section, we run the greedy algorithm on $9$ randomly generated instances with 5, 6, or 7 targets. The procedure for generating the instances is similar to that of the instances generated for simulation. The Table \ref{tab:1} shows the number of potential locations chosen by the algorithm for landmark placement. 
\begin{table}[h!]
    \centering
    \begin{tabular}{c|ccccccccc}
        \toprule
         Instance \# & 1 & 2 & 3 & 4 & 5 & 6 & 7 & 8 & 9  \\
         \midrule
         \# targets & 5 & 5 & 5 & 6 & 6 & 7 & 7 & 7 & 7 \\
         \# landmarks & 21 & 22 & 12 & 24 & 22 & 20 & 31 & 21 & 12 \\
         \bottomrule
    \end{tabular}
    \caption{Number of landmarks placed by the algorithm.}
    \label{tab:1}
\end{table}

\section{Conclusion} \label{sec:conclusion}
This article develops a greedy algorithm for a landmark placement problem to aid localization of vehicles in GPS-denied environment with collision avoidance constraints while taking into account the field of view of the camera fixed to the vehicle. This is the first approach to consider these constraints. The algorithm presented is optimal for each edge the vehicle traverses. Through simulations, it is shown that the landmark placement obtained using our algorithm can aid the vehicle to estimate its position and heading with a high degree of accuracy. Future work would focus on hardware experiments and jointly optimizing the path and the placement along with the constraints considered in this article; this would result in a path for the UV and a corresponding landmark placement that would yield a lower number of landmarks. 

\section{Acknowledgement} \label{sec:acknowledgement}
This work was in part supported by the U.S. Department of Energy through the LANL/LDRD program and the Center for Nonlinear Studies, and in part by the U.S. National Science Foundation (NSF) %and AFRL
through NSF-IIS Robust Intelligence Award 1736087. %and AFRL Grant FA8651-16-1-0001 respectively.

\bibliographystyle{plain}
\bibliography{sensor_placement}
\end{document}